\documentclass[oneside, 11pt]{article}

\usepackage[latin1]{inputenc}
\usepackage[english]{babel}
\usepackage{amsthm}
\usepackage{amsmath}
\usepackage{amsxtra}
\usepackage{mathrsfs}
\usepackage{wasysym}
\usepackage{amssymb}
\usepackage[dvips]{graphicx}
\usepackage{graphicx}

\input xy
\xyoption{all}

\newtheorem{teo}{Theorem}[section]
\newtheorem{lemma}[teo]{Lemma}

\newtheorem{defi}[teo]{Definition}
\newtheorem{coro}[teo]{Corollary}

\newtheorem{prop}[teo]{Proposition}

\theoremstyle{remark}
\newtheorem{remark}[teo]{Remark}

\begin{document}

\begin{center}
{\LARGE{\textbf{Reflexive insensitive modal logics} }}\\
\vspace{0.1 cm}
{\large{David Gilbert}}\\ 
\vspace{0.1 cm}
CAPES post-doctoral fellow \\ 
Centro de 
L\'{o}gica, Epistemologia e Hist\'{o}ria de la Ci\^{e}ncia (CLE), \\ Univ. Est. de Campinas\\
{\large{Giorgio Venturi}\\
\vspace{0.1 cm}
FAPESP post-doctoral fellow \\ 
Centro de L\'{o}gica, Epistemologia e Hist\'{o}ria de la Ci\^{e}ncia (CLE), \\ Univ. Est. de Campinas\\}
\end{center}
\vspace{1 cm}

\begin{abstract}
We analyze a class of modal logics rendered insensitive to reflexivity by way of a modification to the semantic definition of the modal operator. We explore the extent to which these logics can be characterized, and prove a general completeness theorem on the basis of a translation between normal modal logics and their reflexive-insensitive counterparts. Lastly, we provide a sufficient semantic condition describing when a similarly general soundness result is also available.
\end{abstract}

\section{Introduction}\label{sec: Intro}

This paper deals with modal logics that are rendered insensitive to the presence
of reflexivity in the accessibility relation by way of a suitable modification of the standard semantics.
Logics of this kind have already been introduced, independently, by \cite{Marcos} and  \cite{SteinsvoldTop}, 
with the intention of providing formal analyses of certain metaphysical and epistemological notions, respectively.\footnote{In \cite{Marcos}, the focus is on logics of essence and accident. In \cite{SteinsvoldTop}, the analysis is aimed at elucidating the logic of unknown truths.} 

Our intension here is not to provide a critique or endorsement of these interpretations, but rather to give a formal study to these logics (which we will call \emph{reflexive-insensitive}), from the perspective of modal logic.  \cite{Marcos} provided a sound and complete axiomatization of the minimal reflexive-insensitive logic. This result was extended by \cite{Steinsvold}, accounting for the reflexive-insensitive analogs of $\mathbf{T}$, $\mathbf{S4}$ and $\mathbf{S4.3}$. However, both papers lacked a comprehensive treatment of the new semantics, as well as the corresponding modal operator, symbolized by $\circ$.

In this paper we propose a general account of the relationship between normal modal logics and reflexive-insensitive modal logics. We will provide a method to associate a normal modal logic $\mathbf{L}$ with its reflexive-insensitive counterpart, which we will call $\mathbf{L}^{\circ}$.

Our contribution to this subject, therefore, consists in both a conceptual clarification of the notions involved, and in proving general results that describe the conditions under which characterization results for a logic $\mathbf{L}^{\circ}$ follow from the corresponding results for $\mathbf{L}$. In particular, we will prove a general completeness theorem for any logic $\mathbf{L}^{\circ}$, provided that the corresponding normal logic $\mathbf{L}$ is canonical and complete with respect to a class of frames $\mathbb{C}_{\mathbf{L}}$ containing the canonical frame.

Moreover, although a fully general soundness result is not as forthcoming, we will demonstrate that there is a semantic condition, what we will call robustness under reflexivity, that is able to act as a sufficient condition for the logic  $\mathbf{L}^{\circ}$ to be sound with respect to the class of all $\mathbf{L}$-frames.

The paper is organized as follows. Section \ref{sec: Lang&Sem} introduces the $\circ$-operator and the 
corresponding semantics, outlining the phenomenon of insensitivity to reflexivity that can be formally described by way of \emph{mirror reduction} (following the terminology of \cite{Marcos}). We then explain how this property is in fact responsible for almost all of the results contained in \cite{Marcos} and \cite{Steinsvold}. In Section \ref{sec: Minimal} we briefly present the minimal reflexive-insensitive modal logic, following the presentation of \cite{Steinsvold}. In Section \ref{sec: Completeness} the $\circ$-translation is defined, and a general completeness theorem for $\circ$-translations of normal modal logics is proved
using a clever model-theoretic technique from \cite{GoldblattMares}. In Section \ref{sec:Soundness} we define the semantic notion 
of robustness under reflexivity, and we provide soundness results encompassing the $\circ$-translations of many well-known normal modal logics. In Section \ref{sec: Axiomatics}, we will address explicitly the project of axiomatizing logics in the language of $\circ$. Finally, in Section \ref{sec: Conclusion}, we will propose some concluding considerations summarizing the results of this paper in a more abstract setting. We also set the stage for some future work.

\section{Language and Semantics}\label{sec: Lang&Sem}

In this paper we will be working with two languages: the usual language of modal logic, which we will call $\mathcal{L}^\Box$, and the language of the reflexive-insensitive logics, which we call $\mathcal{L}^\circ$. Letting $Var$ be a countable set of propositional variables (we can assume the same set of propositional variables for both languages), the formulas of $\mathcal{L}^\Box$, $Form_{\mathcal{L}^\Box}$,  are defined as usual (for $p \in Var$):

\begin{center}
$\varphi ::= p \mid \neg \varphi \mid \varphi \land \varphi \mid \Box \varphi$
\end{center}

\noindent
and the well-formed formulas of $\mathcal{L}^\circ$, $Form_{\mathcal{L}^\circ}$, are defined recursively as:

\begin{center}
$\varphi ::= p \mid \neg \varphi \mid \varphi \land \varphi \mid \circ \varphi$
\end{center}

\noindent
One can define $\top$, $\bot$, $\lor$, $\to$, $\leftrightarrow$, and $\Diamond$ (in $\mathcal{L}^\Box$) as usual. In the language $\mathcal{L}^\circ$, we also define the operator $\bullet$  so that $\bullet \varphi$ stands for $\neg \circ \varphi$.

\subsection{Semantics}

Structurally, the relational semantics we will use for the logics in both languages are the same, but the clauses in the definition of truth will differ.

\begin{defi}[Frame and Model]
A frame $F$ is an ordered pair $\langle W, R \rangle$, where $W$ is a non-empty set of states and $R \subseteq W \times W$. A model $M = \langle F, V \rangle$ is a frame along with a valuation function $V: Var \to \mathcal{P}(W)$. 
\end{defi}

Intuitively, $V$ assigns, to each variable $p$, the set of states at which $p$ will be considered true. The truth of a formula of $\mathcal{L}^\Box$, with respect to a model and state, is defined as normal. 

\smallskip
\begin{tabular}{lll}
$M, w \models p$ & iff \text{ }& $w \in V(p)$ \\
$M, w \models \neg \varphi$ & iff & $M, w \not\models \varphi$\\
$M, w \models \varphi \land \psi$ \text{ }& iff & $M, w \models \varphi$ and $M, w \models \psi$\\
$M, w \models \Box \varphi$ & iff & $M, x \models \varphi$ for all $x$ s.t.\ $wRx$
\end{tabular}

\smallskip
For the $\mathcal{L}^\circ$ formulas, the propositional formulas are treated identically, and the interpretation of $\circ$ is given by:

\smallskip
\begin{tabular}{lll}
$M, w \models \circ \varphi$ & \text{ } iff \text{ }& either $M, w \not\models \varphi$ or, for all $x \in W$, if $wRx$ then $M, x \models  \varphi$ 
\end{tabular}

Thus, for $\bullet$, we have:

\smallskip
\begin{tabular}{lll}
$M, w \models \bullet \varphi$ & \text{ } iff \text{ }& $M, w \models \varphi$ and there exists an $x \in W$ s.t.\ $wRx$ and $M, x \not\models  \varphi$ 
\end{tabular}

\smallskip
A formula is said to be true in a model $M$ when it is true at every state in $M$. A formula is said to be valid with respect to a frame $F$ when it is true in every model based on $F$, and a formula is valid with respect to a class of frames when it is valid on each frame in the class.

One can view these clauses as providing a unified definition of the truth of a formula, regardless of the language used. That is, when evaluating $\mathcal{L}^\circ$-formulas, one will utilize the semantic condition for $\circ$, but not the one for $\Box$, though one could, if one wished, consider that clause present.  However, despite this, we think it is convenient to use the following notational convention: $M, x \models_{ri} \alpha$ indicates that we are evaluating $\alpha$ in the context of the clauses appropriate for $\mathcal{L}^\circ$, or, equivalently, that $\alpha$ is a formula in the language $\mathcal{L}^\circ$. However, when clear from the context, we will not stress the semantic context. 

It is also worth pointing out that while $\circ \varphi$ can be defined within the context of normal modal logic as $\varphi \to \Box \varphi$, it is not the case that $\Box \varphi$ can always be defined within $\mathcal{L}^\circ$ \cite{Marcos}. In extensions of $\mathbf{T}$, however, it can be regarded as an abbreviation of $\varphi \land \circ \varphi$.\footnote{Below, we will exploit this understanding of $\Box \varphi$ when defining a translation between the languages $\mathcal{L}^\Box$ and $\mathcal{L}^\circ$. And while it is not always a genuine definition, it can be assumed, harmlessly, in a wide variety of cases, which will be detailed.}

\subsection{Mirror Reduction}

Consider the following definition, from \cite[p.~50]{Marcos}.

\begin{defi}[Mirror Reduction]
Let $F = \langle W, R \rangle$ and $F^m = \langle W, R^m \rangle$ be frames such that $R^m \subseteq R$ and $R \setminus R^m \subseteq \{ \langle x, x \rangle : x \in W \}$. Then $F^m$ is said to be a mirror-reduction of $F$. Two frames are said to be mirror-related, $F_1 \sim_m F_2$, when they are  mirror reductions of a common frame. 
\end{defi}

More casually stated, $F^m$ is just the result of removing some reflexive arrows from $F$. Though this is a straightforward concept, its utilization sheds some immediate light on the behaviour of formulas (and, hence, logics) in the language $\mathcal{L}^\circ$ with respect to the semantics outlined above. Immediately, for example, one can obtain the following lemma.

\begin{lemma}[\cite{Marcos}, Lemma 4.2]
Let $F^m$ be a mirror reduction of $F$. Then for any models $M$ and $M^m$, based on $F$ and $F^m$, respectively, and any $x \in W$, 

\begin{center}
$M, x \models_{ri} \alpha$ iff $M^m,x \models_{ri} \alpha$
\end{center}

\noindent
for all $\mathcal{L}^\circ$-formulas $\alpha$.

\end{lemma}

\begin{proof}
This is proved by way of a straightforward induction on the complexity of formulas. We include only the case for $\circ$. 

Assuming $M, x \models_{ri} \circ \varphi$, we have that either $M, x \not\models_{ri} \varphi$ or, for all $y \in W$, if $xRy$ then $M, y \models_{ri}  \varphi$. In the first case, from the induction hypothesis, we have that $M^m, w \not\models_{ri} \varphi$, and so $M^m, x \models_{ri} \circ \varphi$. 

So assume that for all $y \in W$, if $xRy$ then $M, y \models_{ri}  \varphi$. Then, since $R^m \subseteq R$, we have if $xR^my$ then $M, y \models_{ri}  \varphi$, and from the induction hypothesis we obtain  $xR^my$ implies $M^m, y \models_{ri}  \varphi$, as desired.

In the other direction, consider $M^m, x \models_{ri} \circ \varphi$. Again, if $M^m, x \not\models_{ri} \varphi$ then we are finished. So assume that $M^m, x \models_{ri} \varphi$ and that for all $y \in W$, if $xR^my$ then $M^m, y \models_{ri}  \varphi$. From the induction hypothesis we have that if $xR^my$ then $M, y \models_{ri}  \varphi$. But now, since we have that $M^m, x \models_{ri} \varphi$, we also have $M, x \models_{ri} \varphi$, and so if $xRy$ then $M, y \models_{ri}  \varphi$, as desired.

\end{proof}

Perhaps the most interesting, and applicable, aspect of this theorem is the following corollary:

\begin{coro}[\cite{Marcos}, Lemma 4.2]
If $F_1 \sim_m F_2$ then, for all $\mathcal{L}^\circ$-formulas $\alpha$, 

\begin{center}
$F_1 \models_{ri} \alpha$ iff $F_2 \models_{ri} \alpha$.
\end{center}
\end{coro}

In \cite{Marcos}, a sound and complete axiomatization was presented for the simplest logic that is insensitive to reflexivity. That is, an axiom system in the language of $\mathcal{L}^\circ$ was demonstrated to be sound and complete with respect to the class of all frames. Extensions of this logic, however, were not fully explored. For example, in \cite{Marcos} the following open problem was posed: Provide a natural axiomatization for the set of  $\mathcal{L}^\circ$-formulas that are valid on the class of reflexive frames \cite[p.~48]{Marcos}. \cite{Steinsvold} solved this open problem by proving:

\begin{teo}[\cite{Steinsvold}, Proposition 3.5]\label{compT}
Let $\mathbf{KX}$ be any normal modal logic between $\mathbf{K}$ and $\mathbf{KT}$: $\mathbf{K} \subseteq \mathbf{KX} \subseteq \mathbf{KT}$. Then the following are equivalent, for $\alpha$ a formula of $\mathcal{L}^\circ$:

\begin{enumerate}
\item $\alpha$ is valid over the class of all frames;
\item $\alpha$ is valid over the class of all frames for $\mathbf{KX}$;
\item $\alpha$ is valid over the class of all reflexive frames.
\end{enumerate}

\end{teo} 

\noindent
The proof provided by \cite{Steinsvold}, however, was based on the canonical construction (which differed from the one offered by \cite{Marcos}) of the basic logic. The point we would like to make here is that, in fact, this result (and others similar to it) is a direct corollary of the mirror-reduction results given above, as the following proof demonstrates.

\begin{proof}
Clearly, validity for the class of all frames implies validity for the class of $KX$ frames. Similarly for the move from $KX$ to $KT$ frames. Thus, just assume that a formula $\alpha$ is valid in all reflexive frames. We have to show that it is valid in all frames whatsoever. Assume this not to be the case. That is, assume that there is a frame on which one can invalidate $\alpha$. In this case, such a frame is a mirror-reduction of a fully reflexive frame, and so we would have that the fully reflexive frame also invalidates $\alpha$, and this is obviously a contradiction.
\end{proof}

Therefore, by providing an axiomatization for the basic logic, \cite{Marcos} also provided an axiomatization for the logic of all reflexive frames, and all intermediate logics, thereby answering his own question. 

Obviously, such a result can be generalized to some extent. 

\begin{prop}
Let $\mathbf{L} \subseteq \mathbf{LX} \subseteq \mathbf{LT}$ be normal modal logics and $\mathbb{C}_\mathbf{L}$, $\mathbb{C}_{\mathbf{LX}}$, and $\mathbb{C}_{\mathbf{LT}}$ be the classes of $\mathbf{L}$, $\mathbf{LX}$, and $\mathbf{LT}$ frames, respectively. Then, if it is the case that the addition of all possible reflexive arrows to a frame in $\mathbb{C}_\mathbf{L}$ results in a frame in $\mathbb{C}_{\mathbf{LT}}$, then, for any $\alpha \in \mathcal{L}^\circ$, $\alpha$ is valid in $\mathbb{C}_\mathbf{L}$ iff it is valid in $\mathbb{C}_{\mathbf{LX}}$ iff it is valid in $\mathbb{C}_{\mathbf{LT}}$. 
\end{prop}

\noindent
As in the specific case of $\mathbf{K}$ and $\mathbf{T}$, the reason is simply mirror-reduction.

To provide another illustration of this observation, one can consider the logics  $\mathbf{K4} \subseteq \mathbf{K4X} \subseteq \mathbf{K4T} = \mathbf{S4}$. Recall that the class of $\mathbf{K4}$ frames are the transitive frames and  the $\mathbf{S4}$ frames are transitive and reflexive. Then, because the addition of all reflexive arrows preserves transitivity, we immediately obtain the result that if one can axiomatize the reflexive-insensitive logic (in $\mathcal{L}^\circ$) corresponding to the transitive frames, then one also has axiomatized the logic for transitive reflexive frames, and all intermediate logics as well. This result is also contained in \cite{Steinsvold}, where, once again, a canonical model construction is used in the proof once an adequate axiomatization has been provided for $\mathbf{K4}$ (by means of the axiom $(\circ \varphi \land \varphi) \to \circ(\circ \varphi \land \varphi)$). 

On the contrary, if one cannot add reflexive arrows to a frame while preserving the relevant structural properties, then this result clearly does not hold. For example, we can consider the logics situated between $\mathbf{K5}$ and $\mathbf{S5}$.\footnote{$\mathbf{K5}$ is the logic obtained by adding the axiom $\Diamond \varphi \to \Box \Diamond \varphi$ to $\mathbf{K}$.} $\mathbf{K5}$ is characterized by the class of euclidean frames (if $xRy$ and $xRz$ then $yRz$) and $\mathbf{S5}$ by the class of frames in which $R$ is an equivalence relation. However, when one adds reflexive arrows to a euclidean frame one need not obtain an $\mathbf{S5}$ frame, because this might require the extra step of taking the closure (of $R$) under the euclidean condition. Thus, the jobs of axiomatizing these classes of frames, in the reflexive-insensitive setting, are separate.   

The following definition will be useful when proving soundness and completeness for systems, as we will below. It is simply an attempt to formalize the effect of mirror-reduction on soundness and completeness results.

\begin{defi}
Let $\mathbb{C}$ be a class of frames, and let $\mathbf{L}$ be a logic in the language $\mathcal{L}^\circ$. 
Then we say that $\mathbf{L}$ is $m$-characterized by $\mathbb{C}$ if $\mathbf{L}$ is sound and complete with respect 
to $\mathbb{C}^m = \{F : \exists F' \in \mathbb{C}( F \sim_m F')\}$.\footnote{Notice that the notion of $m$-characterization  
is weaker than the standard one, and it is also different from the notion of \textbf{LEA}-characterizability as defined in \cite{Marcos}.} 
\end{defi}

\section{The Minimal Logic}\label{sec: Minimal}

For the remainder of this paper we will be concerned with logics in the language $\mathcal{L}^\circ$ and their relationships with logics in $\mathcal{L}^\Box$. In order to properly define the logics in $\mathcal{L}^\circ$ with which we are concerned, we will make use of the following axiom schemata (found in \cite{Steinsvold}):

\begin{tabular}{ll}
$b0$ \text{ }& $\circ \top$\\
$b1$ & $\bullet \varphi \to \varphi$ \\
$b2$ & $(\circ \varphi \land \circ \psi) \to \circ (\varphi \land \psi)$ \\
\end{tabular}

\noindent
as well as the rule

\begin{tabular}{ll}
$bN$ \text{ }& from $\vdash \varphi \to \psi$ one can obtain $\vdash (\circ \varphi \land \varphi) \to (\circ \psi \land \psi)$. \\
\end{tabular}

\begin{defi}[$RI$-Logics]
An \emph{$RI$-logic} is a set of $\mathcal{L}^\circ$ formulas that contains all substitution instances of propositional tautologies, all instances of the schemata $b0$, $b1$, and $b2$, and is closed under the rules Modus Ponens, $bN$, and Uniform Substitution.
\end{defi}

For now, following \cite{Steinsvold}, we can call the smallest $RI$-Logic $\mathbf{B_K}$.

\begin{prop}\label{BKThms}
The following are theorems of $\mathbf{B_K}$:

\begin{enumerate}
\item $((\circ \varphi \land \varphi) \lor (\circ \psi \land \psi)) \to \circ (\varphi \lor \psi)$
\item $\varphi \lor \circ \varphi$
\item $\varphi \to (\circ(\varphi \to \psi) \to (\circ \varphi \to \circ \psi))$
\end{enumerate}

In addition, the following rules are derivable:

\begin{enumerate}
\item from $\vdash \varphi \leftrightarrow \psi$ one can obtain $\vdash \circ \varphi \leftrightarrow \circ \psi$
\item from $\vdash \varphi$ one can obtain $\vdash \circ \varphi$
\end{enumerate}
\end{prop}

\begin{proof}
We will give a proof of the two rules. Though this result was referred to in \cite{Steinsvold}, an explicit derivation was not provided. 

For the first rule, assume that $\vdash \varphi \leftrightarrow \psi$. We will just prove one direction. The other direction is obtained in exactly the same manner.

\begin{tabular}{llr}
1. \text{ } & $\vdash (\circ \varphi \land \varphi) \leftrightarrow (\circ \psi \land \psi)$ & from the rule $bN$\\
2. & $\vdash (\circ \varphi \land \varphi) \to \circ \psi$ & from line 1\\
3. & $\vdash (\circ \varphi \land \neg \varphi) \to \neg \psi$ & since $\vdash \varphi \leftrightarrow \psi$\\
4. & $\vdash \neg \psi \to \circ \psi$ & because $\vdash \psi \lor \circ \psi$\\
5. & $\vdash (\circ \varphi \land \neg \varphi) \to \circ \psi$ & lines 3 and 4\\
6. & $\vdash ((\circ \varphi \land \varphi) \lor (\circ \varphi \land \neg \varphi)) \to \circ \psi$ & lines 2 and 5\\
7. & $\vdash (\circ \varphi \land (\varphi \lor \neg \varphi)) \to \circ \psi$ & line 6\\
8. & $\vdash \circ \varphi \to \circ \psi$ \\
\end{tabular}

The second rule is just a consequence of the first. If $\vdash \varphi$ then $\vdash \varphi \leftrightarrow \top$, and so, from the first rule, $\vdash \circ \varphi \leftrightarrow \circ \top$. Because $\vdash \circ \top$, we have $\vdash \circ \varphi$, as desired.

\end{proof}

\begin{teo}\label{BKSoundness}
$\mathbf{B}_\mathbf{K}$ is sound with respect to the class of all frames.
\end{teo}

A proof of this result can be found in both \cite{Marcos} and \cite{Steinsvold}.

\section{Completeness}\label{sec: Completeness}

We can prove an immediate completeness result for $\mathbf{B_K}$ by way of a standard canonical model construction. The basic construction is the same as the one in \cite{Steinsvold}. (In light of Theorem \ref{compT}, we will then have that the logic $\mathbf{B_K}$ is sound and complete 
with respect to $\mathbb{C}_\mathbf{T}$, and, in fact, $\mathbb{C}_\mathbf{K}^m$.)  In addition, we will show that this result generalizes to cover a much wider range of modal logics. 

The canonical model $M_{\mathbf{B_K}} = \langle W_{\mathbf{B_K}}, R_{\mathbf{B_K}}, V_{\mathbf{B_K}} \rangle$ is defined as follows:

\begin{itemize}
\item[] $W_{\mathbf{B_K}} := $ the set of all maximal $\mathbf{B_K}$-consistent sets of formulas;
\item[] $R_{\mathbf{B_K}} := \{ \langle x,y\rangle \in W_{\mathbf{B_K}} \times W_{\mathbf{B_K}} : \lambda(x) \subseteq y \}$, for $\lambda(x) := \{ \varphi \in Form_{\mathcal{L}^\circ}: (\varphi \land \circ \varphi) \in x\}$;
\item[] $V_{\mathbf{B_K}}(p) = \{x \in W_{\mathbf{B_K}} \mid p \in x\}$.
\end{itemize}

As a matter of convenience, for the remainder of this section we will omit subscripts.

\begin{remark}
Note, at the outset, that by definition our canonical model is going to be reflexive. That is, it will always be the case that $\lambda(x) \subseteq x$, since if $\varphi \in \lambda(x)$, then it must be that $\varphi, \circ \varphi \in x$.
\end{remark}

\begin{prop}
The relevant version of the Lindenbaum lemma holds. That is, any $B_K$-consistent set of formulas can be extended to a maximal set.
\end{prop}

\begin{lemma}
[\cite{Steinsvold}, Propositions 3.1, 3.2 and 3.3]
The following properties hold of $\lambda(x)$.

\begin{enumerate}
\item $\lambda(x) \neq \emptyset$
\item if $\varphi, \psi \in \lambda(x)$ then $\varphi \land \psi \in \lambda(x)$
\item if $\varphi \in \lambda(x)$ and $\mathbf{B_K} \vdash \varphi \to \psi$, then $\psi \in \lambda(x)$
\end{enumerate}

\end{lemma}

We can then obtain the usual truth lemma.

\begin{lemma}
[Truth Lemma]
For any $\mathcal{L}^\circ$-formula $\alpha$, and any maximal set $w$,

\begin{center}
$M_{\mathbf{B_K}}, w \models_{ri} \alpha$ iff $\alpha \in w$.
\end{center}
\end{lemma}

Again, the proof is in \cite{Steinsvold}.

\begin{teo}
The logic $\mathbf{B_K}$ is strongly complete with respect to the class of all frames. 
\end{teo}

\subsection{Generalized Completeness}

We can generalize the above completeness result significantly. In particular, will show that completeness results for normal modal logics can, under quite general circumstances, be imported into the setting of $RI$-logics.   We require, however, a translation between the formulas of these languages.

\begin{defi}\label{ballification}
Define the following translation from formulas of $\mathcal{L}^\Box$ 
to formulas of $\mathcal{L}^\circ$. 

\begin{tabular}{rll}
$p^\circ$ \text{ }& $=$ \text{ }& $p$ \\
$(\neg \varphi)^\circ$ \text{ }& $=$ & $\neg(\varphi^\circ)$\\
$(\varphi \land \psi)^\circ$ \text{ }& $=$ & $\varphi^\circ \land \psi^\circ$ \\
$(\Box \varphi)^\circ$ \text{ }& $=$ & $\circ (\varphi^\circ) \land \varphi^\circ$ \\
\end{tabular}

In addition, for a normal modal logic $\mathbf{L}$,  define $\mathbf{L}^\circ$ to be the smallest $RI$-logic containing $\varphi^\circ$ for every $\varphi \in \mathbf{L}$.
\end{defi}

\begin{teo}\label{BK=K}
$\mathbf{K}^\circ = \mathbf{B_K}$. That is, $\mathbf{K}^\circ$ is the smallest $RI$-logic.

\end{teo}

\begin{proof}
Clearly, $\mathbf{B_K} \subseteq \mathbf{K}^\circ$. 

In the other direction, we will show that if $\alpha$ is a theorem of $\mathbf{K}$, then $\alpha^\circ$ is a theorem of $\mathbf{B_K}$, and so $\mathbf{K}^\circ \subseteq \mathbf{B_K}$. We can achieve this by way of an induction on proofs.

First, if $\alpha$ is an instance of the $K$ schema, then it is of the form $\Box(\varphi \to \psi) \to (\Box \varphi \to \Box \psi)$. In this case, $\alpha^\circ$ is 

$$(\circ(\varphi^\circ \to \psi^\circ) \land (\varphi^\circ \to \psi^\circ)) \to ((\circ \varphi ^\circ \land \varphi^\circ) \to (\circ \psi^\circ \land \psi^\circ)).$$ 

Assume that this formula is not valid. In this case, there would be a model $M$ and world $w$, at which

$$M,w \models_{ri} \circ(\varphi^\circ \to \psi^\circ) \land (\varphi^\circ \to \psi^\circ)$$

\noindent
but at which 

$$M,w \not\models_{ri} (\circ \varphi ^\circ \land \varphi^\circ) \to (\circ \psi^\circ \land \psi^\circ).$$

It must then be the case that 

$$M,w \models_{ri} \circ \varphi ^\circ \land \varphi^\circ$$

\noindent
while

$$M,w \not\models_{ri} \circ \psi^\circ \land \psi^\circ.$$

Note that it is impossible for $M,w \not\models_{ri} \psi^\circ$, because we have that $M,w \models_{ri} \varphi^\circ$ and also $M,w \models_{ri} \varphi^\circ \to \psi^\circ$.

Therefore, it must then be the case that 

$$M,w \not\models_{ri} \circ \psi ^\circ.$$

From the semantic clause governing $\circ$, this entails that $M,w \models_{ri} \psi^\circ$ but that there exists some $y$ s.t.\ $wRy$ and $M,y \not\models_{ri} \psi^\circ$.

However, in light of the fact that $M,w \models_{ri} \circ \varphi ^\circ \land \varphi^\circ$,  $\varphi^\circ$ must hold at $y$. That is,

$$M,y \models_{ri} \varphi ^\circ.$$

Also, because $M,w \models_{ri} \circ(\varphi^\circ \to \psi^\circ) \land (\varphi^\circ \to \psi^\circ)$, we also have that 

$$M,y \models_{ri} \varphi ^\circ \to \psi^\circ.$$

This gives $M,y \models_{ri} \psi^\circ$, a contradiction. Therefore, all translation instances of the $K$ schema are valid. From the completeness result of $\mathbf{B_K}$ above, they must also be theorems of $\mathbf{B_K}$.

Lastly, $\alpha$ might be the result of applying a rule of inference to some formulas (the translations of which are already in $\mathbf{B_K}$). The cases for Modus Ponens and Uniform Substitution are immediate, from the definition of the $\circ$ translation.

In case $\alpha$ is the result of applying necessitation to some $\beta$, then $\alpha$ is of the form $\Box \beta$. But $(\Box \beta) ^\circ$ is just $\circ \beta ^\circ \land \beta^\circ$. From our assumption we already know that $\beta^\circ \in \mathbf{B_K}$, and so $\circ \beta^\circ$ is also in $\mathbf{B_K}$, from proposition \ref{BKThms}. Therefore, so is $\circ \beta ^\circ \land \beta^\circ$.

This completes the proof that  $\mathbf{K}^\circ \subseteq \mathbf{B_K}$. Therefore,  $\mathbf{K}^\circ = \mathbf{B_K}$.

\end{proof}

In light of this result, we will henceforth refer to the minimal $RI$-logic as $\mathbf{K}^\circ$. Additionally, notice that $(\Box \varphi \to \varphi)^\circ = (\circ (\varphi^\circ) \land \varphi^\circ) \to \varphi^\circ$ is a tautology in $\mathbf{K}^\circ$, and so $\mathbf{T}^\circ = \mathbf{K}^\circ$. In light of Theorem \ref{compT}, this should not be surprising.

By utilizing this translation, we can obtain completeness results for a wide variety of $RI$-logics.

In order to do so, recall the following, standard, definitions and results (see \cite{Blackburn} for all details).

\begin{defi}[Bounded Morphism]
Let $F_1=\langle W_1, R_1 \rangle$ and $F_2=\langle W_2, R_2 \rangle$ be frames. Then $f: W_1 \to W_2$ is a bounded morphism from $F_1$ to $F_2$ when the following two conditions are met:

\begin{itemize}
\item[] (forth) $xR_1y$ implies $f(x)R_2f(y)$;
\item[] (back) if $f(x)R_2z$, then there is a $w$ s.t.\ $xR_1w$ and $f(w)=z$.
\end{itemize}

When there is a surjective bounded morphism from $F_1$ onto $F_2$, written $F_1 \twoheadrightarrow F_2$, $F_2$ is said to be a bounded morphic image of $F_1$.
\end{defi}

\begin{defi}[Generated Subframe]
Let $F_1=\langle W_1, R_1 \rangle$ and $F_2=\langle W_2, R_2 \rangle$ be frames. $F_2$ is a generated subframe of $F_1$, written $F_2 \rightarrowtail F_1$, when $F_2$ is a subframe of $F_1$ and the following condition holds:

\begin{itemize}
\item[] if $x \in W_2$ and $xR_1y$, then $y \in W_2$.
\end{itemize}
\end{defi}

\begin{teo}\label{ValidityPreservation}
Let $F_1$ and $F_2$ be frames and $\alpha$ a modal formula.

\begin{itemize}
\item[] If $F_1 \rightarrowtail F_2$, then $F_2 \models \alpha$ implies $F_1 \models \alpha$;
\item[] If $F_1 \twoheadrightarrow F_2$, then $F_1 \models \alpha$ implies $F_2 \models \alpha$.
\end{itemize}
\end{teo}

\begin{defi}[Canonical Logic]
A normal modal logic $\mathbf{L}$ is said to be canonical when the frame of its canonical model is an $\mathbf{L}$-frame. (That is, when all $\mathbf{L}$-theorems are valid on the canonical frame.)
\end{defi}

Our goal is to prove the following.

\begin{teo}\label{completeness}
Let $\mathbf{L}$ be a normal modal logic that is canonical. Furthermore, let its canonical frame be contained in the class $\mathbb{C}_\mathbf{L}$. Then $\mathbf{L}^\circ$ is also complete with respect to $\mathbb{C}_\mathbf{L}$.
\end{teo}

We will proceed by constructing an isomorphism between the canonical model for  $\mathbf{L}^\circ$ and a generated subframe of the canonical model for $\mathbf{L}$. Specifically, we will construct an injective bounded morphism from the canonical model of $\mathbf{L}^\circ$ to that of $\mathbf{L}$. This is a proof strategy taken from \cite{GoldblattMares}. We only modify their technique slightly, to accommodate for the $\circ$-operator in our logics.

Consider a mapping from $Var$ onto $Form_{\mathcal{L}^\circ}$:

\begin{center}
\begin{tabular}{rll}
$Var$ \text{ }& $\to$ \text{ }& $Form_{\mathcal{L}^\circ}$\\
$p$ \text{ }& $\mapsto$ & $p^*$ 
\end{tabular}
\end{center}

This map exists since our sets of formulas are countable. 
Now extend it recursively to a map: 

\begin{center}
\begin{tabular}{rll}
$Form_{\mathcal{L}^\Box}$ \text{ }& $\to$ \text{ }& $Form_{\mathcal{L}^\circ}$\\
$\alpha$ \text{ }& $\mapsto$ \text{ }& $\alpha^*$ 
\end{tabular}
\end{center}

\noindent
where $\alpha^*$ is defined similarly to Definition \ref{ballification}: 

\begin{tabular}{rll}
$(\neg \varphi)^*$ \text{ }& $=$ \text{ }& $\neg(\varphi^*)$\\
$(\varphi \land \psi)^*$ \text{ }& $=$ & $\varphi^* \land \psi^*$ \\
$(\Box \varphi)^*$ \text{ }& $=$ & $\circ (\varphi^*) \land \varphi^*$ \\
\end{tabular}

Except for the last clause, the above map consists in a uniform substitution of $p^*$ for $p$ in $\alpha$. 
We call the above function the $*$-map. 

Because of how $\mathbf{L}^\circ$ is defined, the $*$-map preserves theoremhood. That is, if $\beta$ is a theorem of $\mathbf{L}$, then $\beta^*$ is a theorem of $\mathbf{L}^\circ$.

Let $F_{\mathbf{L}^\circ} = \langle W_{\mathbf{L}^\circ}, R_{\mathbf{L}^\circ}\rangle$ 
be the canonical frame for $\mathbf{L}^\circ$, as we constructed it before, 
and let $F_{\mathbf{L}} = \langle W_{\mathbf{L}}, R_{\mathbf{L}}\rangle$ be the canonical frame for $\mathbf{L}$ as it is usually defined. 

(As a remark, we notice that if the set of axioms of $\mathbf{L}^\circ$ gives rise to an inconsistent system, then 
the corresponding logic is indeed complete with respect to any class of frames, since everything is provable.
Hence, from now on, we will just assume that $\mathbf{L}^\circ$ is a consistent axiomatic system, and thus the set $W_{\mathbf{L}^\circ}$ is non-empty.)

We can then define the following function: 

\begin{center}
\begin{tabular}{rll}
$f: W_{\mathbf{L}^\circ}$ \text{ }& $\to$ \text{ }& $W_{\mathbf{L}}$\\
$a$ \text{ }& $\mapsto$ & $\{\alpha : \alpha^* \in a\} = f(a)$ 
\end{tabular}
\end{center}

\noindent
for any maximal consistent $a \in W_{\mathbf{L}^\circ}$. In order to show that the above is a meaningful definition, we have to verify that $f(a)$ is indeed an element of $W_{\mathbf{L}}$.

\begin{prop}
The set $f(a)$ is maximal and $\mathbf{L}$-consistent. 
\end{prop}

\begin{proof}
For what concerns consistency, assume not. Then there are formulas $\alpha_1, \ldots, \alpha_n \in f(a)$
such that $$\mathbf{L} \vdash (\alpha_1 \land \ldots \land \alpha_n) \to \bot.$$  
As a consequence of $*$ preserving theoremhood we can infer that 
$$\mathbf{L}^\circ \vdash (\alpha_1^* \land \ldots \land \alpha_n^*) \to \bot$$  

\noindent
with $\alpha_1^*, \ldots,  \alpha_n^* \in a$. This contradicts the consistency of $a \in W_{\mathbf{L}^\circ}$.

For maximality, again assume not. Then there is a formula $\alpha \in Form_{\mathcal{L}^\Box}$ such that 
neither $\alpha$ nor $\neg \alpha$ is in $f(a)$. As a consequence, according to the definition of $f(a)$, 
neither $\alpha^*$ nor $\neg \alpha^*$ is in $a$. This contradicts the maximality of  $a \in W_{\mathbf{L}^\circ}$.
\end{proof}

Therefore, the definition of $f$ makes sense. We now show that $f$ is actually an injective bounded morphism. We proceed by means of the following claims. 

\begin{prop}
The function $f$ is injective.
\end{prop}
\begin{proof}
Assume $a \neq b$. We must show $f(a) \neq f(b)$. Without loss of generality, we may assume that there is a formula $\theta \in a \setminus b$. 
Since $\theta$ belongs to $Form_{\mathcal{L}^\circ}$, it is equal to some $p^*$, for $p \in Var$. Therefore, by maximality of $b$, we have that 
$\neg \theta \in b$. Now, since $\theta = p^*$ we also have that $\neg \theta = \neg (p^*) = (\neg p)^*$. As a consequence, $p \in f(a)$
and $\neg p \in f(b)$, thus showing that $f(a) \neq f(b)$. 
\end{proof}

\begin{prop}
If $a R_{\mathbf{L}^\circ} b$ then $f(a) R_{\mathbf{L}} f(b)$.
\end{prop}

\begin{proof}
The claim consists in showing that $\Box^-\big{(}f(a)\big{)}\subseteq f(b)$. 
To this aim, assume $\alpha \in \Box^-\big{(}f(a)\big{)}$, and so  $\Box \alpha \in f(a)$. 
Then, by definition of $f(a)$, $(\Box\alpha)^* \in a$. By definition of the $*$-translation, 
we conclude that $\circ (\alpha^*) \land \alpha^* \in a$. This means, in particular, 
that $\alpha^* \in \lambda(a)$, and so  $\alpha^* \in b$, by our hypothesis. 
Hence, $\alpha \in f(b)$, by definition of $f(b)$. 
\end{proof}

\begin{prop}
If $f(a) R_{\mathbf{L}} c$ ,then there is a $b \in W_{\mathbf{L}^\circ}$ such that 
$a R_{\mathbf{L}^\circ} b$ and $f(b) = c$. 
\end{prop}

\begin{proof}
Define the following set:
$$
b_0 = \{\alpha^* : \alpha^* \land \circ(\alpha^*) \in a \} \cup \{\beta^* : \beta \in c\}.
$$
We claim that $b_0$ is consistent. Assume not, and let $\alpha^*, \beta^* \in Form_{\mathcal{L}^{\circ}} \cap b_0$
such that 
$$\mathbf{L}^{\circ} \vdash \alpha^* \land \beta^* \to \bot.$$
Hence, we have the following deductions.
\begin{enumerate}
\item $\mathbf{L}^{\circ} \vdash \alpha^* \to \neg \beta^*$
\item $\mathbf{L}^{\circ} \vdash \alpha^* \to \big{(} \circ(\alpha^* \to \neg \beta^*) \to (\circ \alpha^* \to \circ \neg \beta^*)  \big{)}$
\item $\mathbf{L}^{\circ} \vdash \circ(\alpha^* \to \neg \beta^*)$
\end{enumerate}

\noindent
where (2) is an instance of a theorem of $\mathbf{L}^{\circ}$, as pointed out in Fact \ref{BKThms}. 
Moreover, since $\alpha^* \in a$ we can show that $\circ \neg \beta^* \in a$. 
As a consequence, $\neg \beta^* \land \circ \neg \beta^* \in a$. Thus $\big{(} \Box(\neg \beta)\big{)} ^* \in a$,
and so $\Box \neg \beta \in f(a)$. By our hypothesis we then have that $\neg \beta \in c$, thus contradicting the consistency of $c$. 

Now extend $b_0$ to a maximal set and name it $b$. We have to show that $\lambda(a) \subseteq b$ and that $f(b) = c$. 

So assume $\alpha \in \lambda(a)$. Since $\alpha$ is a formula in $\mathcal{L}^\circ$, we know that there is a 
$p \in Var$ such that $p^* = \alpha$. Hence $p^* \land \circ (p^*) \in a$ and so, by construction, $p^* \in b_0 \subseteq b$.
Thus $\alpha \in b$. 

In order to show that $f(b) = c$ it is sufficient to notice that, by construction, $c \subseteq f(b)$. 
And so the equality holds by the maximality of $c$.

\end{proof}

At this stage, we have shown that $f$ is indeed an injective bounded morphism from the canonical frame of $\mathbf{L}^\circ$ to that of $\mathbf{L}$. Furthermore, the image of $f$ is a generated subframe of the canonical frame of $\mathbf{L}$, and is isomorphic to the canonical frame of $\mathbf{L}^\circ$.\footnote{Clearly, since we are considering the image of $f$, we obtain a bijection between the canonical frame of $\mathbf{L}^\circ$ and a subframe of the canonical frame for $\mathbf{L}$. The fact that this bijection is in fact an isomorphism follows from two applications of Theorem \ref{ValidityPreservation}: in the one direction we consider $f$, and in the other $f^{-1}$, both of which are surjective bounded morphisms. Lastly, the fact that the subframe is a generated subframe of $F_{\mathbf{L}}$ is a consequence of the (back) condition placed on $f$.} Symbolically, we have 

$$F_{\mathbf{L}^\circ} \cong F_{sub} \rightarrowtail F_{\mathbf{L}}$$

\noindent
(where $F_{sub}$ is the subframe of $F_{\mathbf{L}}$). Therefore,  $F_{\mathbf{L}^\circ}$ is actually an $\mathbf{L}$-frame, from Theorem \ref{ValidityPreservation}.

Finally, assume that some formula $\alpha$ is not a theorem of $\mathbf{L}^\circ$. Then, clearly, it is not valid on the canonical frame $F_{\mathbf{L}^\circ}$. In turn, we then know that there is a generated subframe of $F_{\mathbf{L}}$, call it $F_{sub}$, on which $\alpha$ is not valid (since $F_{\mathbf{L}^\circ} \cong F_{sub}$). This then implies that $\alpha$ is not valid on  $F_{\mathbf{L}}$ (because $F_{sub} \rightarrowtail F_{\mathbf{L}}$ and so $F_{\mathbf{L}} \models \alpha$ implies $F_{sub} \models \alpha$). Therefore, on the assumption that $\mathbf{L}$ is canonical, we have that $\mathbf{L}^\circ$ is complete with respect to classes of frames containing the canonical frame of $\mathbf{L}$, as desired.

This completes the proof of Theorem \ref{completeness}.

\begin{coro}
Let $\mathbf{L}$ be a normal modal logic that is canonical. Furthermore, let its canonical frame be contained in the class $\mathbb{C}_\mathbf{L}$. Then $\mathbf{L}^\circ$ is complete with respect to $\mathbb{C}_\mathbf{L}^m$.
\end{coro}

\section{Soundness}\label{sec:Soundness}

In this section we will give a sufficient semantic condition for a logic $\mathbf{L}^\circ$ to be sound with respect to $\mathbb{C}_\mathbf{L}$. Though we do not obtain a single soundness theorem that is as general as our completeness theorem, we do obtain a result that covers a wide variety of normal modal logics and their $\circ$-translations. In conjunction with Theorem \ref{completeness}, this then provides an $m$-characterization theorem for those logics satisfying the condition.

Before starting, we need a lemma that connects the truth of $\mathcal{L}^\Box$-formulas  and $\mathcal{L}^\circ$-formulas.

\begin{lemma}\label{bridge}
Let $M = \langle F, V \rangle$ be a model based on $F =\langle W, R \rangle$ and let
$\alpha$ be a formula of the language $\mathcal{L}^{\Box}$. Then, for every $x \in W$,
the following holds:

$$
M, x \models_{ri} \alpha^\circ \iff M^r, x \models \alpha
$$

\noindent
where $M^r$ stands for the model $\langle F^r, V \rangle$, based on the frame $F^r = \langle W, R^r \rangle$, given by $R^r = R \cup \{(x, x) : x \in W\}$. 

\end{lemma}

\begin{proof}
We prove the lemma by induction on the complexity of $\alpha$. 
If $\alpha = p \in Var$, then the result is immediate, since the valuations in the two models are identical. 

If $\alpha = \beta \land \gamma$, then an easy application of the inductive hypothesis 
shows that the conclusion of the lemma holds.

If $\alpha = \neg \beta$, then $M^r, x \models \neg \beta$ iff $M^r, x \not\models \beta$ iff 
$M, x \not\models_{ri} \beta^\circ$ iff $M, x \models_{ri} \neg \beta^\circ$.

Finally, if $\alpha = \Box \beta$, then $M^r, x \models \Box \beta$ implies
that for all $y \in W$ s.t.\ $xR^ry$, $M^r, y \models \beta$. 
By the inductive hypothesis, this is equivalent to saying that for all $y \in W$, $xR^ry$ implies $M, y \models_{ri} \beta^\circ$. From this we obtain that for all $y \in W$ s.t.\ $xRy$, $M, y \models_{ri} \beta^\circ$ and that $M, x \models_{ri} \beta^\circ$. Thus, we have that  $M, x \models_{ri} \circ \beta^\circ \land \beta^\circ$, and so $M, x \models_{ri} (\Box \beta)^\circ$. 

In the other direction, if  $M, x \models_{ri} (\Box \beta)^\circ$ then $M, x \models_{ri} \circ \beta^\circ \land \beta^\circ$. We then have that $M, x \models_{ri} \beta^\circ$ and so $M^r, x \models \beta$. In addition, from $M, x \models_{ri} \circ \beta^\circ$ we have that for all $y$ s.t.\ $xRy$,  $M, y \models_{ri} \beta^\circ$ (since the other possibility, that $M, x \not\models_{ri} \beta^\circ$, has been ruled out). Then, from the induction hypothesis we get that for all $y$ s.t.\ $xRy$,  $M^r, y \models \beta$. But, since $M, x \models_{ri} \beta$, we have that for all $y$ s.t.\ $xR^ry$,  $M^r, y \models \beta$, and so  $M^r, y \models \Box \beta$.

\end{proof}

\begin{defi}
We will say that a class of frames $\mathbb{C}$ is robust with respect to reflexivity when the following condition holds:

\begin{center}
If $F \in \mathbb{C}$, and $F^r$ is the result of adding all reflexive arrows to $F$, then $F^r \in \mathbb{C}$.
\end{center}

\end{defi}

In other words, $\mathbb{C}$ is robust with respect to reflexivity when the reflexive closure of each frame in $\mathbb{C}$  is also in $\mathbb{C}$ .

Notice that this is obviously not equivalent to saying that $F$ and $F^r$ are mirror related. First, $F^r$ is obtained, specifically, by adding arrows. In addition, $F^r$ is completely reflexive.

\begin{teo}\label{SufficientSoundness}
Let $\mathbf{L}$ be a normal modal logic that is sound with respect to a class of frames $\mathbb{C}_\mathbf{L}$ that is robust with respect to reflexivity. Then $\mathbf{L}^\circ$ is sound with respect to $\mathbb{C}_\mathbf{L}$. In fact, $\mathbf{L}^\circ$ is sound with respect to $\mathbb{C}_\mathbf{L}^m$.
\end{teo}

\begin{proof}
We will show that for every theorem $\varphi$ of $\mathbf{L}$, $\varphi^\circ$ is valid on $\mathbb{C}_\mathbf{L}$. Since the rules of $\mathbf{K}^\circ$ preserve validity, this will imply that every theorem of $\mathbf{L}^\circ$ is valid on $\mathbb{C}_\mathbf{L}$. Assume, for a contradiction, that this is not the case. Thus, there exists a frame $F \in \mathbb{C}_\mathbf{L}$, such that $F \not\models \varphi^\circ$. 

Thus, there is a model $M$, based on $F$, and a state $x$, such that $M, x \not\models_{ri} \varphi^\circ$. From lemma \ref{bridge}, we then have that $M^r, x \not\models \varphi$. Therefore, $F^r \not\models \varphi$. But if $\mathbb{C}_\mathbf{L}$ is robust to reflexivity, it would have to be that $F^r \in \mathbb{C}_\mathbf{L}$, and so $F^r \models \varphi$, a contradiction.

Therefore, $\mathbf{L}^\circ$ is sound with respect to $\mathbb{C}_\mathbf{L}$, and also with respect to $\mathbb{C}_\mathbf{L}^m$.

\end{proof}

Though this result lacks the generality present in the completeness theorem, there are still some immediate corollaries.

\begin{coro}
Let $\mathbf{L}$ be any normal modal logic extending $\mathbf{T}$, and let $\mathbb{C}_\mathbf{L}$ be the class of all $\mathbf{L}$-frames. Then $\mathbf{L}^\circ$ is sound with respect to $\mathbb{C}_\mathbf{L}$.
\end{coro}

\begin{proof}
If $\mathbf{L}$ extends $\mathbf{T}$, then any $\mathbf{L}$ frame is reflexive. Therefore, $\mathbf{C}_\mathbf{L}$ is obviously robust with respect to reflexivity. Thus, $\mathbf{L}^\circ$ is sound with respect to $\mathbb{C}_\mathbf{L}$.
\end{proof}

We can also apply this theorem in order to obtain more specific results. The next corollary lists just some examples of this, and is in no way comprehensive.

\begin{coro}
The following soundness results hold:

\begin{enumerate}
\item $\mathbf{D}^\circ$ is sound with respect to the class of all serial frames;
\item $\mathbf{K4}^\circ$ is sound with respect to the class of all transitive frames;
\item $\mathbf{KB}^\circ$ (where $\mathbf{KB} = \mathbf{K}+(\varphi \to \Box \Diamond \varphi)$) is sound with respect to the class of all symmetric frames;
\item $\mathbf{KM}^\circ$ (where $\mathbf{KM} = \mathbf{K}+(\Box\Diamond\varphi \to \Diamond \Box \varphi)$) is sound with respect to the class of all final frames.\footnote{However, note that in this case we do not have a completeness result because $\mathbf{KM}^\circ$ is not canonical \cite{Goldblatt}.}
\end{enumerate}
\end{coro}

\begin{proof}
The classes of serial, transitive, symmetric, and final (every state is related to at least one state that is related only to itself) frames are all robust with respect to reflexivity.
\end{proof}

Note, in addition, that soundness is going to be preserved when combining these logics, as usual. That is, for example, we have that $\mathbf{KB4}^\circ$ is sound with respect to the class of transitive symmetric frames. Thus, while the soundness result is less general than desired, in fact one can still use it to obtain soundness results for a surprisingly wide range of normal modal logics.

However, we have already encountered one system that sits outside the scope of the soundness theorem: $\mathbf{K5}^\circ$. Recall that $\mathbf{K5}$ is characterized by the class of euclidean frames. However, euclidean frames are not robust to reflexivity. To take a trivial example, one can consider the frame in which $W = \{x,y\}$ and $R = \{\langle x,y \rangle, \langle y, y, \rangle\}$. On this frame, the euclidean condition is vacuously satisfied. However, when one adds all reflexive arrows, we obtain the frame with the accessibility relation $R^r = \{\langle x,y\rangle, \langle x,x \rangle, \langle y, y\rangle \}$. This is no longer euclidean, as $xRy$ and $xRx$ ought to imply that $yRx$, but we lack this relationship. The point, therefore, is that our soundness theorem does not, on the basis of a soundness theorem for $\mathbf{K5}$, provide us with a theorem for the translated logic $\mathbf{K5}^\circ$. And, in fact, it is straightforward to construct a euclidean frame that does not validate $5^\circ$, the translation of $\Diamond \varphi \to \Box \Diamond \varphi$.

\section{Axiomatizing $RI$-Logics}\label{sec: Axiomatics}

Our results so far place conditions on when $\mathbf{L}^\circ$ will be sound and complete with respect to the class of frames $\mathbb{C}_\mathbf{L}$. So far, we have not explicitly mentioned the issue of axiomatizing these logics, a topic that was very central to both \cite{Marcos} and \cite{Steinsvold}. We can say something about this now.

As the following theorem demonstrates, in order to obtain an axiomatization of $\mathbf{L}^\circ$, one can simply take any adequate axiomatization of $\mathbf{L}$, and add the translations of these axioms to $\mathbf{K}^\circ$. This is, more or less, a consequence of the definition of $\mathbf{L}^\circ$. Moreover, as in the case of normal modal logics, the choice of axiomatization does not matter.

\begin{teo}
Let $\mathbf{L}$ be a normal modal logic that is axiomatized by adding some axiom $A$ to $\mathbf{K}$. Let $\mathbf{K}^\circ + A^\circ$ be the smallest $RI$-logic that contains all instances of $A^\circ$. Then $\mathbf{K}^\circ + A^\circ = \mathbf{L}^\circ$.
\end{teo}

\begin{proof}
Clearly, since $A \in \mathbf{L}$, $\mathbf{K}^\circ + A^\circ \subseteq \mathbf{L}^\circ$, from the definition of $\mathbf{L}^\circ$.

In the other direction assume that $\alpha \in \mathbf{L}^\circ$ but that $\alpha \not\in \mathbf{K}^\circ + A^\circ$.

There are two options regarding $\alpha$: either it is the translation of some $\beta$ that is a theorem of $\mathbf{L}$, or else it is a product of rule applications.

In the first case, since $\beta \in \mathbf{L}$, and $\mathbf{K}+A$ is assumed to be an adequate axiomatization of $\mathbf{L}$,  $\mathbf{K} \cup \{A\} \vdash \beta$. However, this would then imply that $\mathbf{K}^\circ \cup\{ A^\circ\} \vdash \beta^\circ$, since the application of rules in $\mathbf{L}$ is honored by the translation, as was demonstrated as part of the proof of Theorem \ref{BK=K}. This would then be a contradiction, as $\alpha$ is $\beta^\circ$.

In the second case, $\alpha$ is the result of the application of rules to some finite set of formulas $B = \{\beta_1,\dots,\beta_n\}$, where each $\beta_i$ ($1\leq i \leq n$) is either an instance of $b0$, $b1$, or $b2$, or the translation of some $\gamma \in \mathbf{L}$. However, as we have just demonstrated, it would have to be that for any such $\gamma$ we have that $\gamma ^\circ \in \mathbf{K}^\circ + A^\circ$. Since all instances of $b0$, $b1$, and $b2$ are also obviously in $ \mathbf{K}^\circ + A^\circ$, and $\mathbf{L}^\circ$ and $ \mathbf{K}^\circ + A^\circ$ are closed under the same rules, $\alpha$ must be in  $\mathbf{K}^\circ + A^\circ$, as desired.

\end{proof}

An immediate corollary of this result is that if a logic $\mathbf{L}$ is axiomatized by two different axiomatizations, then the translations of these axiomatizations, in the above sense, both provide axiomatizations of $\mathbf{L}^\circ$, as one would hope.

\section{Concluding Remarks}\label{sec: Conclusion}

We may describe the $\circ$-translation as  a functor between $\mathfrak{N}$, 
the collection of all normal modal logics, and $\mathfrak{N}^\circ$, the
collection of all non-normal modal logics in the language $\mathcal{L}^\circ$
that extend $\mathbf{K}^\circ$.

\begin{center}
\begin{tabular}{rll}
$F: \mathfrak{N}$ \text{ }& $\to$ \text{ }& $\mathfrak{N}^\circ$\\
$\mathbf{L}$ \text{ }& $\mapsto$ & $\mathbf{L}^\circ$ 
\end{tabular}
\end{center}

\noindent As the results of the previous sections show, the behavior of $F$ may be useful
in understanding the meta-theoretical properties of members of $\mathfrak{N}^\circ$.

We might then reformulate Theorem \ref{compT} by saying 
that $F(\mathbf{T}) = \mathbf{K}^\circ$ and that
the logic $\mathbf{T}^\circ$ is $m$-characterized by $\mathbb{C}_\mathbf{K}$. 
In the same way, Proposition 3.6 in \cite{Steinsvold}
can be expressed saying that $F(\mathbf{K4}) = F(\mathbf{S4})$, and that 
the logic $\mathbf{S4}^\circ$ is $m$-characterized by $\mathbb{C}_\mathbf{K4}$.  
Moreover, notice that since $\mathbf{S5} = \mathbf{K5} + T$ and $T^\circ$ is a tautology,
we have that  $F(\mathbf{K5}) = F(\mathbf{S5})$. However, 
the logic $\mathbf{S5}^\circ$ is not $m$-characterized by $\mathbb{C}_\mathbf{K5}$.
As a consequence, our method is not able to give a straightforward axiomatization of a 
logic in $\mathfrak{N}^\circ$ able to  be $m$-characterized by $\mathbb{C}_\mathbf{K5}$.

A possible development of this work---which may be of independent interest in the study of normal modal logics---is
the possibility of giving a syntactic characterization of the semantic notion of robustness with respect to reflexivity.  
Indeed we believe that,  at more general level, the topics and the results 
of this paper illustrate the potential usefulness of utilizing non-normal modal logics in the pursuit of a better understanding of normal ones. We hope that the content 
and the techniques of this paper will help foster the analysis of logics with different modal operators that are able to give new insights into the meta-theoretical study of normal modal logics.

In particular, one might undertake an extensive study of a $\star$-operator, 
whose definition is complementary with respect to that of the $\circ$-operator:

\medskip

\begin{tabular}{lll}
$M, w \models \star \varphi$ \text{ } & iff \text{ }& either $M, w \models \varphi$ or, for all $x \in W$, if $wRx$ then $M, x  \models  \varphi$ 
\end{tabular}

\medskip

\noindent 
A first step in the study of logics that may be called reflexive intolerant has  already been made in 
\cite{SteinsvoldW}, in the context of epistemic logic.\footnote{\cite{SteinsvoldW} performs a study of a slightly different operator, namely $W \varphi = \Box \varphi \land \neg \varphi$.}  We intend to study this further in future work.

\end{document}